\documentclass[12pt]{article}
\usepackage{latexsym}
\usepackage{amsfonts}
\usepackage{amsmath}
\usepackage{amsmath,amssymb,proof1,latexsym}
\usepackage[all]{xypic}

\makeindex

\newcommand{\forces}{\!\Vdash\!}
\newcommand{\imp}{\!\rightarrow\!}

\newcommand{\proves}{\vdash}

\newtheorem{Prop}{\bf Proposition}
\newenvironment{proposition}{\begin{Prop}\em }{\end{Prop}}
\newtheorem{Theor}{\bf Theorem}

\newtheorem{Lemma}{\bf Lemma}

\newtheorem{Coro}{\bf Corollary}

\newtheorem{Fact}{\bf Fact.}

\newtheorem{Remark}{\bf Remark}

\newtheorem{Claim}[enumi]{Claim}

\newtheorem{defin}{\bf Definition}
\newenvironment{definition}{\begin{defin} \em}{\end{defin}}
\newtheorem{exam}{\bf Example}
\newenvironment{example}{\begin{exam} \em}{\end{exam}}
\newtheorem{notat}{\bf Notation.}

\newenvironment{proof}{{\bf Proof.}}{\hfill $\slot$}
\newcommand{\slot}{\hfill \mbox{$\Box$}\vspace{\parskip}\\}
\newtheorem{Comment}{\bf Comment}

\newcommand{\bk}{\ensuremath{\mathbf{K}}}

\begin{document}

\title{Knowing the Model}

\author{Sergei Artemov\\ \\
 {\small The City University of New York, the Graduate Center}\\
{\small  365 Fifth Avenue, New York City, NY 10016}\\
{\small {\tt sartemov@gc.cuny.edu}} }
\date{\today}
\maketitle

\begin{abstract}  Epistemic modal logic normally views an epistemic situation as a Kripke model. We consider a more basic approach: 
to view an epistemic situation as a set $W$ of possible states/worlds -- maximal consistent sets of propositions -- with conventional accessibility relations determined by $W$.  We find that in many epistemic situations, $W$ is not a Kripke model: a necessary and sufficient condition for $W$ to be a Kripke model is the so-called {\it fully explanatory property} -- a propositional form of {\it common knowledge of the model} -- which has been a hidden (and overlooked) assumption in epistemic modal logic. 

We sketch a theory that describes epistemic models in their generality. We argue for conceptual and practical value of new models, specifically for representing partial knowledge, asymmetric knowledge, and awareness. 
\end{abstract}

\medskip\par
\section{Preliminaries}\label{prelim}
In this note we will try to present things at both levels, conceptual and technical. On the formal side, we will focus on the propositional $n$-agent epistemic logic  ${\sf S5}^n$, cf. \cite{FHMV95} though all the major findings and suggestions apply to other modal logics as well. Furthermore, similar considerations apply to other classes of epistemic models, e.g., Aumann structures \cite{Aum95,Aum99}.

Informally, by a global state\footnote{In this text we will also be using terms {\it state} or {\it world} for global states, when convenient.} of a multi-agent system we understand a complete description of epistemic states of agents along with state of nature, represented as a set of propositions in an appropriate epistemic language. 

\begin{definition}\label{estate}
In a formal setting, a {\it  global state} is a maximal consistent (over a given logic base, among which ${\sf S5}^n$ is the default) set of formulas. An {\it epistemic model} $(W,\models)$  is a pair of a set $W$ of global states and truth assignment to formulas at each world
\[ w\models F\ \ \mbox{\it iff}\ \ F\in w .\]
\end{definition}

Each epistemic model $(W,\models)$ determines relations $R_i$ of epistemic accessibility for each agent: $uR_i v$ if and only if whatever agent $i$ knows in $u$ is true in $v$. Therefore, each epistemic model $(W,\models)$ has an {\it associated Kripke model}, cf. Definition~\ref{Kripke}, with the same atomic evaluation $\Vdash$ as in $(W,\models)$ 
\[ u\forces p\ \ \mbox{\it iff}\ \ p\in u .\]
Truth assignments in an epistemic model and the associated Kripke model coincide for the atomic propositions, but can differ for compound formulas, cf. Example~\ref{e1}.

Our starting point is an epistemic model $(W,\models)$. We do not analyze the origins of the worlds from $W$. In order to study the problem in its full generality, we {\bf do not assume} that $W$ is a conventional Kripke model (a standard assumption in formal epistemology). This reflects our conceptual approach that a (commonly known) Kripke structure is not \emph{a priori} superimposed on intellectual agents and that the analysis 
of epistemic scenarios should start at the earlier point, with an epistemic model $(W,\models)$. 

\subsection{A Brief Summary}
Given an epistemic model $(W,\models)$, we attempt to analyze the following issues.

\begin{enumerate}
\item {\bf Whether all epistemic models are Kripke models with the same states under the standard accessibility/indistinguishability relation ``everything known at $w$ is true at $v$."}  \underline{The traditional approach in epistemic modal logic treats} \underline{epistemic models as if the answer is ``yes," whereas it is ``no."} There are epistemic models (intuitively, most of them) which are not Kripke models. This simple observation opens the door to studying epistemic models in their full generality, hence closing the loophole in the foundations of epistemic modeling. 
\item {\bf Which epistemic situations can be represented as Kripke models with the standard accessibility relations?}  We show that a necessary and sufficient condition for $(W,\models)$ to be a Kripke model is the so-called {\it fully explanatory property} -- a propositional form of \underline{{\it common knowledge of the model} -- which has been a hidden} \underline{(and overlooked) assumption in epistemic modal logic.} 
\item {\bf What is a conceptual and practical value of the new broader class of models?} We present a body of examples to make a case for (general) epistemic models (Definition~\ref{estate}) as formalizations of  epistemic scenarios. 
\item {\bf Are we suggesting replacing Kripke models with new types of epistemic models?} No. Kripke models continue to play a fundamental technical role in the specification of general epistemic models. The key observation connecting epistemic models and Kripke models is the embedding theorem (Theorem~\ref{two}) stating that \underline{each epistemic model is a sub-model of an appropriate Kripke model.} This suggests a universal ``scaffolding" method of building general epistemic models: take an appropriate Kripke model and carve out a desired subset of states which the knowers consider possible. 
Our findings do not discriminate against Kripke models but rather suggest considering a more general class of models to capture more epistemic subtleties such as partial knowledge, asymmetric knowledge, awareness, and others.
\end{enumerate}

\subsection{Motivations}

We quote \cite{Wil15} for the standard approach to motivate Kripke models in epistemology\footnote{Analyzing the role of knowing the model, normally assumed and not acknowledged in formal epistemology, has been long overdue. The paper that prompted completing this study was \cite{Wil15}.}:
\begin{quote}
{\em 
Informally, we interpret $W$ as a set of mutually exclusive, jointly exhaustive worlds or states, 
$\ldots$ $R$ is a relation of epistemic accessibility: a world w has R to a world x if and only if $\ldots$  whatever the agent knows in $w$ is true in $x$. We define a function {\bf K} from propositions to propositions by
the following equation for all propositions $p$:
\begin{equation}\label{defK}
{\bf K}p = \{w\in W  :  \forall x\in W, wRx \Rightarrow x\in p\}.
\end{equation}

In other words, ${\bf K}p$ is true at a world if and only if $p$ is true at every world
epistemically accessible from that one.  Informally, ${\bf K}p$ is interpreted as the proposition that the agent knows $p$.
}
\end{quote}
Formally, the characterization of $R$ via knowledge at the states in $W$ is 
\begin{equation}\label{defR}
R(w) = \{x\in W :  \forall p, \ {\bf K}p\in w \Rightarrow p\in x \}.
\end{equation}
This yields 
\begin{equation}\label{straight}
{\bf K}F\in w\ \ \ \Rightarrow\ \ \mbox{\it for all $x\in R(w)$},\ F\in x .
\end{equation} 
However, this does not guarantee the converse:
\begin{equation}\label{converse}
(\mbox{\it for all $x\in R(w)$},\ F\in x)\ \ \Rightarrow\ \ {\bf K}F\in w, 
\end{equation}
which is built into definition (\ref{defK}). 
Conceptually, the fact that $F$ holds at some designated set of states should not automatically yield knowledge of $F$ at a given state. 

Technically, equations (\ref{defK}) and (\ref{defR}) do not match. Given knowledge assertions at states of the model, we indeed can find accessibility relation $R$ by (\ref{defR}), and then determine the knowledge modality $\bk$ by (\ref{defK}). The problem is that this $\bk p$ is different from the original knowledge assertion ``$p$ is known."  
\begin{example}{\bf [Technical]}\label{e1}
Model ${\cal M}_1$ in Fig.~1. Consider {\sf S5} with a single propositional letter $p$. Consider also $W$ consisting of one state $w$ generated by $\Gamma=\{p,\neg {\bf K}p\}$.\footnote{State $w$ is constructive: one can check that $\Gamma$ is a complete set of formulas, i.e., for each $F$, either $\Gamma$ proves $F$ or $\Gamma$ proves $\neg F$, and $w$  is the set of formulas derivable from $\Gamma$, cf. also Section~\ref{canon}.} 

\begin{figure}[!h]
\begin{center}
\mbox{
\begin{xy}
(-15,-4)*{w};
(-5,0)*{\mbox{$p, \neg\bk p$}};
(-15,0)*+{\bullet}="1"; 
{\ar@(ul,ur) "1";"1"};
\end{xy}
}\end{center}
\caption{Model $\mathcal{M}_1$.}
\end{figure}
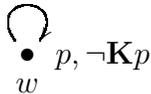
\noindent
The accessibility relation $R$, (\ref{defR}), is reflexive, $wRw$, the truth in the model is membership in $1$: $F$ holds iff $F\in w$. In particular, ${\bf K}p$ is false at $w$. On the other hand, by (\ref{defK}), ${\bf K}p$ ought to be true at $w$. So definitions (\ref{defR}) and (\ref{defK}) do not match in ${\cal M}_1$.
\end{example}
\begin{example}{\bf [Conceptual]}\label{e2}
Kripke models are a convenient vehicle for specifying worlds: each node in a model yields a specific maximal consistent set of formulas. The downside of Kripke specification is that in order to model ignorance of a fact $F$, one has to commit to a hypothetic world at which $F$ is false. Such a world may not exist. 
\begin{quote} {\it Imagine a world $w$ at which an educated agent knows the axioms of Peano Arithmetic {\sf PA}, but does not know a theorem $F$ for which a proof has not yet been found. In a Kripke model, we then have to have a world $v$ deemed possible by the agent at which $\neg F$ holds. However, there cannot be such a consistent world $v$ because all axioms of {\sf PA} should be true at $v$ and ${\sf PA}\cup \{\neg F\}$ is inconsistent. 
}
\end{quote}
A possible way out of this predicament is by epistemic models which naturally allow $F$ to be true at each possible world but yet remain unknown; this is not allowed by Kripkean standards. 
\end{example}

\section{Motivating Epistemic Example}\label{motiv}
In an epistemic situation, a proposition $F$ may be not known for different reasons, e.g., 
\begin{enumerate}
\item there is a state deemed possible by the agent at which $F$ fails; 
\item the agent is unaware of a sufficient justification for $F$;
\item the agent is not aware of $F$. 
\end{enumerate}
\par\noindent
Kripke models fairly represent 1 but are not very good at modeling 2 and 3. If $F$ is true but unknown, a Kripke model is forced to have a hypothetical world in which $F$ fails, and this is not always possible even theoretically. Situation 2 has been studied within the framework of Justification Logic (\cite{Art01a,Art08,Art12,AF15,Fit05}). An example of situation 3 follows.

\begin{example}\label{AB}
{\it Bob is tossing a coin. Ann knows that there are two possibilities, ``heads" or ``tails," but is unaware of an additional condition $Q$ (e.g., the actual bets are so high that she cannot afford to play this game without additional insurance). Condition $Q$ holds at both states but is unknown to Ann. 
}
\end{example}

A Kripkean approach suggests ignoring $Q$ and offers a model ${\cal M}_2$\footnote{In this and other models, cycles of $R$ around states and other redundant $R$-arrows are suppressed for better readability.} which shows Ann's possible states:  this model is {\bf known} to Ann, though she does not know which state is real. This good old Kripke model specifies the truth value of any modal formula with atoms {\it heads} and {\it tails}, but is not faithful to the story, since condition $Q$ is missing. 

\vskip10pt
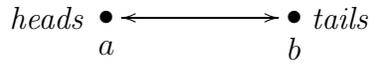
\begin{figure}[!h]
\begin{center}
\mbox{
\begin{xy}
(10,-4)*{b};
(-15,-4)*{a};
(-23,0)*{\mbox{\it heads}};
(16,0)*{\mbox{\it tails}};
(10,0)*+{\bullet}="2";
(-15,0)*+{\bullet}="1"; 
{\ar "1";"2"};
{\ar "2";"1"};
\end{xy}
}\end{center}
\caption{Model $\mathcal{M}_2$.}
\end{figure}

Another possible Kripke model, $\mathcal{M}_3$, incorporates $Q$ and $\neg \bk Q$ by imagining possible states $c$ and $d$ in which $Q$ fails. 

\begin{figure}[!h]
\begin{center}
\mbox{
\begin{xy}
(8,3)*{b};
(-13,3)*{a};
(7,-17)*{d};
(-12,-17)*{c};
(-25,0)*{\mbox{\it heads, Q}};
(20,0)*{\mbox{\it tails, Q}};
(-23,-20)*{\mbox{\it heads}};
(17,-20)*{\mbox{\it tails}};
(10,0)*+{\bullet}="2";
(-15,0)*+{\bullet}="1"; 
(10,-20)*+{\bullet}="4";
(-15,-20)*+{\bullet}="3";
{\ar "1";"2"};
{\ar "2";"1"};
{\ar "1";"2"};
{\ar "2";"1"};
{\ar "1";"3"};
{\ar "3";"1"};
{\ar "2";"4"};
{\ar "4";"2"};
\end{xy}
}\end{center}
\caption{Model $\mathcal{M}_3$.}
\end{figure}
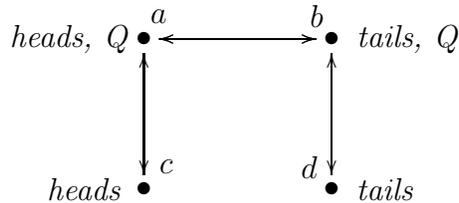
\noindent
Conceptually, in $\mathcal{M}_3$, states $c$ and $d$ are fictional: they are not considered possible by Ann, since Ann is not aware of $Q$ and certainly does not envision possible states $c$ and $d$. $\mathcal{M}_3$ describes a different situation in which  Ann is aware of $Q$'s relevance, but does not know whether $Q$ holds. So, model $\mathcal{M}_3$ is not adequate either: within $\mathcal{M}_3$, Ann should rush to purchase special insurance before playing this game. This is not, however, faithful to the story according to which Ann is unaware of $Q$, does not deem $c$ and $d$ possible (and does not bother about insurance). 

It appears that an adequate epistemic model should have states $a$ and $b$ as in $\mathcal{M}_2$, have $Q$ holding at both states, but yet not known to Ann. 
\vskip10pt
\begin{figure}[!h]
\begin{center}
\mbox{
\begin{xy}
(10,-3)*{b};
(-15,-3)*{a};
(-33,0)*{\mbox{\it heads, Q, $\neg\bk Q$}};
(26,0)*{\mbox{\it tails, Q, $\neg\bk Q$}};
(10,0)*+{\bullet}="2";
(-15,0)*+{\bullet}="1"; 
{\ar "1";"2"};
{\ar "2";"1"};
\end{xy}
}\end{center}
\caption{An incomplete attempt to build Kripke model $\mathcal{M}_4$.}
\end{figure}
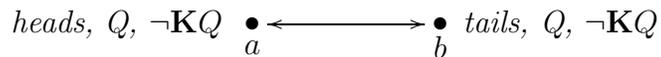
The problem here is that the resulting structure, though faithful to the story, cannot be viewed as a Kripke model. Since $Q$ holds everywhere in $\mathcal{M}_4$, $\bk Q$ should also hold, which is not the case.\footnote{One can easily recognize a version of Moore's Paradox \cite{Moo42} here.} 

In the rest of this note we sketch a theory of new epistemic models that go beyond Kripke models and are capable of representing this and other situations which are normally off the scope of conventional Kripke modeling.

\section{Epistemic Models}
We start from an epistemic model $(W,\models)$, cf. Definition~\ref{estate}, and consider accessibility relations $R_i$ as well as truth/forcing relation $u\forces F$ of a formula at a given world defined by $(W,\models)$, cf. Definition~\ref{Kripke}. 

It appears that there is a tacit assumption in Kripkean formal epistemology\footnote{as presented in \cite{Wil15}} that the induced truth/forcing relation $u\forces F$ coincides with the original ``natural" truth $u\models F$ in $W$.  As we have already noticed in Section~\ref{prelim}, this assumption does not always hold. 

\begin{definition}\label{Kripke}
For each epistemic model $(W,\models)$ over an $n$-agent logical language with knowledge modalities $\bk_1,\ldots,\bk_n$, we define {\it accessibility relations} $R_1,\ldots,R_n$ as
\begin{equation}\label{defRi}
R_i(w) = \{x\in W :  \forall F, \ {\bf K}_i F\in w \Rightarrow F\in x \}.
\end{equation}
The forcing relation ``$\Vdash$" is standard:
\[ u\forces p\ \ \mbox{\it iff}\ \ p\in u .\]
This defines \emph{a Kripke model $(W,R_1,\ldots,R_n,\Vdash)$ associated with $(W,\models)$. }
\end{definition}
\begin{definition}\label{FE}
An epistemic model  is {\em fully explanatory}\footnote{The name {\it fully explanatory} was introduced by Mel Fitting \cite{Fit05} in similar, but slightly different, circumstances.} if the relations $R_i$  satisfy (\ref{converse}):
\[ 
(\mbox{\it for all $x\in R_i(w)$},\ F\in x)\ \ \Rightarrow\ \ {\bf K}_iF\in w .
\] 
\end{definition}
In an epistemic model, once $F$ is known at $u$ (i.e., ${\bf K}_iF\in u$), then $F$ holds everywhere in $R_i(u)$. 
For fully explanatory models, the converse also holds: each $F$ which is true at all states accessible from $u$ ought to be known at $u$. Informally, in fully explanatory models, a proposition $F$ can hold in a set of states $R_i(w)$ only for a reason, namely when $F$ is known to agent $i$ at $w$. 
\begin{example}
A model ${\cal M}_5$ is yet another example of a non-fully explanatory model in which $W$ is the set of all maximal consistent sets over ${\sf S5}^n$ containing a fixed propositional atom $p$ with standard canonical accessibility relations (\ref{defRi}). 
It is immediate that $p$ holds everywhere in $W$. Consider a maximal consistent extension $w$ of a consistent set $\{ p, \neg {\bf K}_1p\}$. Obviously, $p$ holds at $R_1(w)$, and ${\bf K}_1p$ does not. 
\end{example}
\begin{proposition}\label{one} {\it The canonical ${\sf S5}^n$-model is fully explanatory.}
\end{proposition}
\begin{proof} Let $W$ be the set of all maximal consistent sets of formulas over ${\sf S5}^n$, and the canonical accessibility relations (\ref{defRi}). The claim that the truth assignment ``$\models$" in epistemic model $(W,\models)$ coincides with the truth in $(W,R,\Vdash)$ as a Kripke model is reflected in the standard Truth Lemma for ${\sf S5}$: for any $u\in W$ and formula $F$, 
\[ F\in u\ \ \Leftrightarrow\ \ u\forces F . \]
\end{proof}

Kripke models are exactly {\bf fully explanatory epistemic models}. Here is an informal\footnote{It appears that a natural formalization of this condition leads us beyond the current level of propositional modal logic.} sufficient condition under which an epistemic model is a Kripke model (a fully explanatory model): \[ \mbox{\it  Kripke models are epistemic models {\bf commonly known to all agents}.}\] 
Once $F$ holds everywhere in $R(u)$, the agent knows this and, knowing $R$, can conclude that $F$ holds at all states epistemically possible in $u$, thus coming to justified (by virtue of this argument) knowledge of $F$.  So, in Kripke models, knowledge of $F$ at $u$ given $F$ holds in $R(u)$, does not appear unjustified from nowhere. A justification for such a knowledge 
\[ u\forces {\bf K}F \]
is merely assumed knowledge of the model itself relativized to a specific state $u$. 

Accidentally, the reason why the canonical model ${\it CM}({\sf S5}^n)$ of {\sf S5} is fully explanatory is somewhat different: ${\it CM}({\sf S5}^n)$ is a saturated model which includes all consistent states. Since $\neg {\bf K}F$ yields that $\neg F$ is consistent, for each $\neg {\bf K}F\in u$, $R(u)$ contains a state $v$ with $\neg F\in v$, which suffices for the fully explanatory property. 

We will further discuss the ontological status of epistemic models vs. Kripke models in Section~\ref{fromemtokm}.

\section{On the Structure of Epistemic Models} 

Let $(W,\models)$ be an epistemic model and $R_i$'s its induced accessibility relations (\ref{defRi}). 

\begin{proposition}\label{equiv} Each {\it $R_i$ is an equivalence relation on $W$.} 
\end{proposition} 
\begin{proof} Let $\bk$ denote any of $\bk_i$'s  and $R$ any of $R_i$'s. Reflexivity and transitivity are immediate. Let us check symmetry. Let $wRx$, and suppose ${\bf K}F\in x$. We have to prove that $F\in w$. Suppose $F\not\in w$, then, by reflexivity in ${\sf S5}^n$, ${\bf K}F\not\in w$, hence $\neg{\bf K}F\in w$. By negative introspection, ${\bf K}\neg{\bf K}F\in w$. By definition of $R$, $\neg {\bf K}F\in x$, which is impossible since $x$ is consistent and ${\bf K}F\in x$.
\end{proof} 

The intuition of indistinguishability for states from $R(w)$ in epistemic models is similar to Kripke models: we can interpret $R(w)$ as a set of states indistinguishable from $w$ by facts known to the agent. Apparently, $w\forces {\bf K}F$ yields $R(w)\forces F$ and all facts known at $w$ are true everywhere in $R(w)$. So, a state $x\in R(w)$ cannot be distinguished from $w$ by any fact known to the agent. 

The only, but principal, difference between epistemic models and Kripke models is that in the former, a validity of $F$ in $R(w)$ does not yield knowledge of $F$: there is room for ignorance of agents about valid facts\footnote{We regard this as a feature that makes epistemic models more flexible and realistic.}. In particular, it is possible to have $F$ throughout $R(w)$, but $\neg {\bf K}F$ at each state in $R(w)$. 

The following proposition shows that knowledge assertions respect indistinguishability: either $\bk F$ holds everywhere in $R(w)$, or $\neg\bk F$ holds everywhere in $R(w)$. 

\begin{proposition}{\it $R(w) \forces {\bf K}F$ or $R(w) \forces \neg {\bf K}F$.}
\end{proposition}
\begin{proof} Suppose $R(w) \forces {\bf K}F$. Then for some $x\in R(w)$, $x\forces \neg {\bf K}F$. By negative introspection, $x\forces {\bf K}\neg {\bf K}F$, hence $R(x)\forces \neg {\bf K}F$. Since, by Proposition~\ref{equiv}, $R(x)=R(w)$, $R(w) \forces \neg {\bf K}F$. 
\end{proof}

\section{Derivations from Hypotheses in Modal Logic}\label{derGamma}

The standard formulation of {\sf S5} postulates the Necessitation rule: 
\[ \proves F\ \ \Rightarrow\ \ \proves {\bf K}_i F .
\] 
However, this rule is not valid in a general setting for {\sf S5}-derivations from assumptions: for some $\Gamma$, $\Gamma \proves F$ does not yield 
$\Gamma \proves {\bf K}_i F$. Therefore, when speaking about derivations from hypotheses in ${\sf S5}^n$, we do not postulate Necessitation.

\begin{definition}\label{fromgamma}
For a given set of formulas $\Gamma$ (here called ``hypotheses'' or ``assumptions") we consider {\em derivations from $\Gamma$}: assume all ${\sf S5}_n$-theorems together with $\Gamma$ and use classical reasoning (rule {\em Modus Ponens}).  The notation 
\[ 
\Gamma\proves A 
\] 
represents `$A$ is derivable from $\Gamma$.' 
\end{definition}

It is important to see the role of Necessitation in reasoning without assumptions and in reasoning from a nonempty set of assumptions. In the former, the rule of Necessitation is not postulated: if $A$ follows from $\Gamma$, we cannot conclude that $A$ is known, since $\Gamma$ itself can be unknown. However, for some ``good'' $\Gamma$'s, Necessitation is a valid rule.

\begin{example}\label{e0} Consider the case of two agents, ${\sf S5}_2$. 
If we want to describe a situation in which proposition $m$ is known to agent 1, we consider the set of assumptions $\Gamma$:
$$\Gamma = \{\bk_1 m\}.$$ 
From this $\Gamma$, by reflection principle $ \bk_1 m \imp m$ from ${\sf S5}_n$, 
 we can derive $m$, 
 \[ \Gamma\proves m .\]
However, we cannot conclude that agent 2 knows $m$\footnote{An easy a counter-model.}:
\[ \Gamma\not\proves\bk_2 m .
\]
Therefore, there is no Necessitation in this $\Gamma$, since we have $\Gamma\proves m$ but $\Gamma\not\proves \bk_2 m$. 
\end{example}

\section{Canonical Models for {\sf S5} with a Single Letter}\label{single}

In this section we will offer a useful elaborate example of canonical model constructions associated to {\sf S5} with a single propositional letter $p$, ${\sf S5}(p)$. 

We first note that the modality-free fragment generated by $\{p\}$, i.e., the usual classical propositional logic with a single propositional letter $p$, admits two possible worlds: one generated by $\{p\}$ and the other generated by $\{\neg p\}$. 

We claim that ${\sf S5}(p)$ admits exactly four possible worlds (maximal consistent sets):
\begin{itemize}
\item 
$A$, generated by $\{\bk p\} (= \{p, \bk p\})$;
\item
$B$, generated by $\{p, \neg \bk p\}$;
\item
$C$, generated by $\{\neg p, \neg\bk \neg p\}$;
\item
$D$, generated by $\{\bk \neg p\} (= \{\neg p, \bk \neg p\})$.
\end{itemize}

{\bf Consistency of each of $A$--$D$} is straightforward since each has an easy Kripke model. Now we check that each of $A$--$D$ is complete, i.e., that each proves $F$ or $\neg F$ for any formula $F$ in the language of ${\sf S5}(p)$.

{\bf Completeness} of $A$. First we note that $A$ is closed under Necessitation: $A\proves F$ yields $A\proves \bk F$. Standard induction on derivations of $F$. The key point here is that $A\proves \bk A.$ Once we establish Necessitation in $A$, we proceed to proving that for each $F$, $A\proves F$ or $A\proves \neg F$. Induction on $F$. Obvious for atomic formulas and Boolean connectives. Let $F=\bk X$. If $A\proves X$, then, by Necessitation,  $A\proves \bk X$. If $A\proves \neg X$, then, by reflexivity, $A\proves \neg\bk X$.

Completeness of $B$. Here Necessitation is not admissible since $B\proves p$, but $B\not\proves \bk p$. We will use the {\sf S5}-normal forms, cf. \cite{MvdH95}. 
\begin{Lemma}[{\sf S5} normal forms] In {\sf S5}, every formula is provably equivalent to a formula in normal form which is a disjunction of conjunctions of type
\begin{equation}\label{normalf}
\alpha\wedge \bk \beta \wedge \neg\bk\gamma_1\wedge\ldots\wedge\neg\bk\gamma_m 
\end{equation}
where $\alpha,\beta,\gamma_1,\ldots,\gamma_m$ are all purely propositional formulas. For ${\sf S5}(p)$ we may assume that each of them is from $\{\top,\bot, p, \neg p\}$. 
\end{Lemma}
It now suffices to check that for each formula $F$ of type (\ref{normalf}), $B\proves F$ or $B\proves \neg F$. 

If $\alpha=\bot,\neg p$, then $B\proves \neg F$. 

If $\alpha=\top, p$, then $B\proves\alpha$ and we proceed to $\beta$. 

If $\beta=\bot, \neg p$, then, by reflexivity,  $B\proves \neg F$. 

If $\beta=p$, then again,  $B\proves \neg F$. 

If $\beta=\top$, then $B\proves\bk \beta$ and we proceed to $\gamma_i$. 

If at least one of $\gamma_i$ is from $\top$, then $B\proves \neg F$. 

Otherwise, all disjuncts in $F$ are provable in $B$. Indeed, for $\gamma_i=\bot$, use $B\proves \neg\bk\bot$. For $\gamma_i=p$ use the fact that $\neg\bk p\in B$. For $\gamma_i=\neg p$, use reflexivity $p\imp\neg\bk\neg p$. In either case, $B\proves \neg\bk\gamma_i$\footnote{A similar normal form-based proof of completeness can be given for each of $A$--$D$, but we have opted for Necessitation-based proof for $A$ and $D$ to underline the fact that both $A$ and $D$ enjoy Necessitation.}.

Completeness of $C$. Similar to $B$ 

Completeness of $D$. Similar to $A$, since $C$ also enjoys Necessitation. 

The collection of $A$--$D$ {\bf exhausts} all logical possibilities for states over ${\sf S5}(p)$. Indeed, for the remaining four logical possibilities for $p$ and knowledge assertion about $p$, $\{p,\bk \neg p\}$  and $\{\neg p,\bk p\}$ are inconsistent and so are any of its extensions. The last two options: $\{p,\neg\bk\neg p\}\subset A$ and $\{\neg p,\neg\bk p\}\subset C$, hence they generate no new states. 

Finally, we describe the {\bf accessibility relation $R$} on the collection $W=\{A,B,C,D\}$ of all possible states over ${\sf S5}(p)$ as in (\ref{defR}).

By Proposition~\ref{equiv}, $R$ is an equivalence relation on $W$. Consider all six possible pairs of different states in $W$ and rule out the ones that are not accessible from each other. 

Pairs $\{A,C\}$, $\{A,D\}$ are not connected by definition of $R$ since for $(A,X)$ to be in $R$, $p$ should be in $X$, which rules out $\{A,C\}$, $\{A,D\}$. Likewise, $\{B,D\}$ are not connected. 

Pairs $\{A,B\}$ and $\{C,D\}$ are not connected due to positive introspection, e.g., since $\bk p\in A$, $\bk\bk p\in A$ too, hence $(A,X)\in R$ yields $\bk p\in X$; this rules out $\{A,B\}$. 

The only remaining possibility for $R$-connection is pair $\{B,C\}$, and they are related! There should be a state in $W$ accessible from $B$ in which $\neg p$ holds, and $C$ is the only remaining possibility, hence $(B,C)\in R$. 
 The resulting picture of the canonical model for ${\sf S5}(p)$ is
 
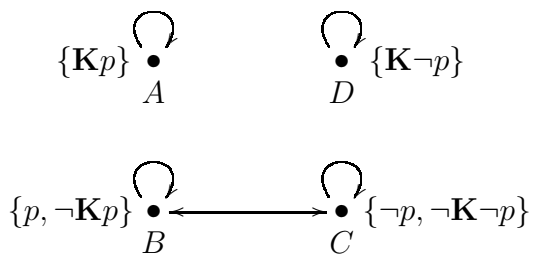
\begin{figure}[!h]
\begin{center}
\mbox{
\begin{xy}
(10,-4)*{C};
(-15,-4)*{B};
(-26,0)*{\mbox{$\{p, \neg\bk p\}$}};
(20,20)*{\mbox{$\{\bk \neg p\}$}};
(10,0)*+{\bullet}="2";
(-15,0)*+{\bullet}="1"; 
{\ar "1";"2"};
{\ar "2";"1"};
{\ar@(ul,ur) "1";"1"};
{\ar@(ul,ur)"2";"2"};
(10,16)*{D};
(-15,16)*{A};
(-23,20)*{\mbox{$\{\bk p\}$}};
(24,0)*{\mbox{$\{\neg p, \neg \bk \neg p\}$}};
(10,20)*+{\bullet}="A";
(-15,20)*+{\bullet}="D"; 
{\ar@(ul,ur) "A";"A"};
{\ar@(ul,ur)"D";"D"};
\end{xy}
}\end{center}
\caption{Canonical model for ${\sf S5}(p)$.}
\end{figure}

\begin{proposition} {\it Each of fifteen non-empty subsets of $\{A,B,C,D\}$ is an epistemic model. Seven of them are fully explanatory and hence Kripke models: 
\[  \{A\},\ \{D\},\ \{A,D\},\ \{B,C\},\ \{A,B,C\},\ \{B,C,D\},\ \{A,B,C,D\}. \]
The remaining eight are epistemic models which are non-fully explanatory and thus not Kripke models. Among them is $\{A,B\}$ --  all states at which $p$ holds; this may be regarded as the canonical model of $\Gamma=\{p\}$ which is therefore not a Kripke model. 
}
\end{proposition}
For a general theory of canonical models for sets of assumptions $\Gamma$ cf. Section~{\ref{canon}.

\section{From Epistemic Models to Kripke Models}\label{fromemtokm}

Each Kripke model $(W,R,\Vdash)$, regarded as a set of states (maximal consistent sets of formulas), is an epistemic model $(W,\Vdash)$. The converse obviously does not hold: there are epistemic models $(W,\models)$ which are not Kripke models with respect to the induced accessibility relation $R$, cf.(\ref{defR}). In this regard, epistemic models conceptually and technically extend the toolbox for epistemic modeling without rejecting any of the ``old" models. 

Kripke models are a convenient vehicle for specifying worlds: each node in a model yields a specific maximal consistent set of formulas. The downside of Kripke specification is that in order to model ignorance of a fact $F$, one has to commit to a hypothetic state at which $F$ is false. As we have already noted, such a state may not exist. 

In this section we show that each epistemic model is a sub-model of an appropriate Kripke model.

\begin{Theor}\label{two} For any epistemic model $(W,\models)$ with induced accessibility relations $R_i$ (\ref{defRi}), there is a Kripke model $(\widetilde{W},\widetilde{R_1},\ldots,\ \widetilde{R_n},\Vdash)$ such that 

a) $W\subseteq \widetilde{W}$ (in particular, for each $u\in W$ and each $F$, $u\models F$ iff $u\forces F$); 

b) $R_i\subseteq \widetilde{R_i}$. 

\end{Theor} 
\begin{proof} We establish a uniform (nonconstructive) version of Proposition \ref{two}. Take 
$$(\widetilde{W},\widetilde{R_1},\ldots,\ \widetilde{R_n},\Vdash)$$ to be the canonical model 
${\it CM}({\sf S5}^n)$: each $u\in W$ is also a world in ${\it CM}({\sf S5}^n)$ and in both models, $F$ holds at $u$ iff $F\in u$. For (b) it suffices to note that the definitions of $uR_iv$ and $u\widetilde{R_i}v$ coincide. 
\end{proof}
\begin{example}
Here is the finite example of such an embedding. The aforementioned epistemic model ${\cal M}_1$ from Example~\ref{e1} is naturally embedded into Kripke model ${\cal M}_6$. A singleton model ${\cal M}_1$, as a world, coincides with world $w$ in ${\cal M}_5$. 

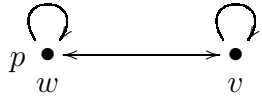
\begin{figure}[!h]
\begin{center}
\mbox{
\begin{xy}
(10,-4)*{v};
(-15,-4)*{w};
(-19,-1)*{\mbox{$p$}};
(10,0)*+{\bullet}="2";
(-15,0)*+{\bullet}="1"; 
{\ar "1";"2"};
{\ar "2";"1"};
{\ar@(ul,ur) "1";"1"};
{\ar@(ul,ur)"2";"2"};
\end{xy}
}\end{center}
\caption{Model $\mathcal{M}_5$.}
\end{figure}
\noindent
We see here that the Kripke model requires two worlds to emulate a singleton epistemic model. Sometimes the blow-up of the number of worlds when embedding an epistemic model to a Kripke model is infinite: it is easy to build a singleton epistemic model for which the corresponding Kripke model is necessarily infinite. 
\end{example}

In view of Theorem~\ref{two}, why do we need Epistemic Models when we can specify the same states by means of conventional Kripke Models? 
Epistemic models and Kripke models provide different kinds of analysis and, in a way, complement each other. 

A Kripke model of a set of assumptions $\Gamma$ models a world containing $\Gamma$ by emulating negative knowledge assertions $\neg {\bf K}F$ geometrically,  throwing in hypothetical worlds to represent the nested ignorance of the agent. Sometimes this approach is intuitive, and helpful, but sometimes the number of auxiliary hypothetical worlds is excessive and their ontological status is dubious (cf. Example~\ref{e2}). These hypothetical worlds can become inconsistent when $\neg F$ contradicts underlying assumptions about the knowledge base of agents. 

Epistemic models of a set of assumptions $\Gamma$ aim rather at describing a set of worlds compatible with $\Gamma$. In an epistemic model for $\Gamma$, all worlds contain $\Gamma$, and the matter of representing negative knowledge assertions by additional hypothetical worlds, possibly not compatible with $\Gamma$, is ignored as an unnecessary technicality which can distort the epistemic picture. 

\begin{definition} Let $\Gamma$ be a set of formulas. By $\Gamma \models F$ we understand the situation when for each epistemic model ${\cal M}$ and its state $u$, 
\[ {\cal M}, u \models\Gamma\ \ \Rightarrow\ \ \ {\cal M}, u \models F .\] 

\end{definition}

\begin{proposition}
{\it Soundness and completeness of ${\sf S5}^n$ w.r.t. epistemic models: $$\Gamma \proves F\ \ \mbox{iff}\ \ \Gamma \models F.$$ }
\end{proposition}
\begin{proof}
Soundness. If $\Gamma \proves F$, then $F$ belongs to any maximal consistent extension of $\Gamma$ and hence is true at each state of each epistemic model of $\Gamma$.

Completeness follows from Kripke completeness: if $\Gamma\not\proves F$, then there is a Kripke model ${\cal M}$ and its state $u$ such that 
${\cal M}, u \forces \Gamma$ and ${\cal M}, u \not\forces F$. This state $u$ may itself be regarded as an epistemic model in which $\Gamma$ holds and $F$ does not. 
\end{proof}

\section{Canonical Epistemic Models in a General Setting}\label{canon}

Kripke models constitute a comprehensive semantical tool in modal logic: every consistent configuration is realized in an appropriate node of the canonical model. However, in modal logic we don't normally care of how ``possible" the states in a Kripke models are, their formal consistency is sufficient. Epistemic scenarios are different. We can imagine a situation in which some propositions $\Gamma$ (e.g., reflecting the state of nature) should hold at all possible states. Furthermore, if this $\Gamma$ is not common knowledge the corresponding set of states is, generally speaking, not a Kripke model, and requires a broader approach. 

Suppose we are interested in epistemic models of a given set of assumptions $\Gamma$. Taking a Kripke model in which $\Gamma$ holds at some state 
\[ {\cal M},u \forces \Gamma \]
gives us a singleton epistemic model corresponding to world $u$ in ${\cal M}$. Other worlds in ${\cal M}$ may be inconsistent with $\Gamma$. 
If we want to study a set of all worlds compatible with $\Gamma$, we might wish to consider canonical epistemic models of sets of assumptions $\Gamma$ over a given modal logic, here {\sf S5}. 

\begin{definition}
A {\it canonical epistemic model of a set of formulas $\Gamma$}, ${\it CM}(\Gamma)$ is, by definition, the collection of all worlds containing $\Gamma$. \end{definition}

Canonical models (in a variety of disguises) play a pivotal role in establishing completeness theorems and are used for other purposes as well. 

We show that for many (intuitively, most) $\Gamma$'s, the corresponding canonical model ${\it CM}(\Gamma)$ is not fully explanatory, hence not a Kripke model. We give a criterion of when the canonical model of $\Gamma$ is a Kripke model: iff \emph{$\Gamma$ is closed under Necessitation} iff \emph{$\Gamma$ proves its own common knowledge}. 

\begin{example} Canonical model  ${\it CM}(p)$ for $\Gamma = \{p\}$ in {\sf S5} has been described in Section~\ref{single}.  There are two possible worlds in ${\it CM}(p)$, generated by $\{{\bf K}p\}$ (world $A$) and $\{p, \neg{\bf K}p\}$ (world $B$).  Worlds $A$ and $B$ are not connected by the relation $R$, $p$ holds at both worlds, but is not known in $B$. The canonical model ${\it CM}(p)$ is not a Kripke model, since $p$ holds in ${\it CM}(p)$, but ${\bf K}p$ does not.
\end{example} 

\subsection{Common knowledge and Necessitation}
In this section we consider a representative case of two agents. 
We will use abbreviations: for ``everybody's knowledge''
\[ {\bf E} X = {\bk_1}X\wedge {\bk_2}X, \]
and ``common knowledge'' 
\[ {\bf C}X = \{X,\ {\bf E}X,\ {\bf E}^2 X,\  {\bf E}^3 X,\ \ldots  \}.  \] 
As one can see, ${\bf C} X$ is an infinite (though quite regular and decidable) set of formulas. Since modalities $\bk_i$ commute with the conjunction $\wedge$, ${\bf C}X$ is the set of all formulas which are $X$ prefixed by iterated knowledge modalities:
\[ {\bf C}X= \{P_1P_2\ldots P_k X \mid k=0,1,2,\ldots,\ \  P_i \in \{ \bk_1,\bk_2\}\}.\]
Naturally, $${\bf C}\Gamma=\bigcup \{{\bf C}F\mid F\in\Gamma\}$$ that represents ``$\Gamma$ is common knowledge.'' The following proposition states that the rule of Necessitation corresponds to derivable common knowledge of assumptions. 

\begin{proposition}\label{ck=nec}  {\em A set of formulas $\Gamma$ is closed under Necessitation if and only if $\Gamma\proves {\bf C}\Gamma$, i.e., that $\Gamma$ proves its own common knowledge.}
\end{proposition}
\begin{proof} Direction `if.' Assume $\Gamma\proves {\bf C}\Gamma$ and prove by induction on derivations that $\Gamma\proves X$ yields $\Gamma\proves \bk_i X$. For $X$ being from ${\sf S5}_n$, this follows from the rule of Necessitation in ${\sf S5}_n$. For $X\in\Gamma$, it follows from the assumption that $\Gamma\proves {\bf C}X$, hence $\Gamma\proves \bk_i X$. If $X$ is obtained from {\em Modus Ponens}, $\Gamma\proves Y\imp X$ and $\Gamma\proves Y$. By IH, $\Gamma\proves \bk_i(Y\imp X)$ and $\Gamma\proves \bk_i Y$. By the distributivity principle of ${\sf S5}_n$, $\Gamma\proves \bk_i X$.

For `only if,' suppose that $\Gamma$ is closed under Necessitation and $F\in\Gamma$, hence $\Gamma\proves F$. Using appropriate instances of the Necessitation rule in $\Gamma$ we can derive $P_1P_2P_3,\ldots,P_k F$ for each prefix $P_1P_2P_3,\ldots,P_k$ with $P_i$ is one of $\bk_1,\bk_2$. Therefore, $\Gamma\proves {\bf C}F$ and $\Gamma\proves {\bf C}\Gamma$.
\end{proof}

\begin{definition} Let $\Gamma, \Delta$ be sets of ${\sf S5}^n$-formulas. We say that {\it $\Delta$ is consistent over $\Gamma$} if $$\Gamma,\Delta \not\proves \bot.$$ 
\end{definition}

\subsection{Canonical Models of Sets of Assumptions}

We answer the question of when the canonical model of $\Gamma$ is a Kripke model. 

\begin{Theor}\label{main}  The following are equivalent:

a) ${\it CM}(\Gamma)$ is fully explanatory (i.e., a Kripke model); 

b) $\Gamma$ admits Necessitation;

c) $\Gamma$ proves its own common knowledge. 
\end{Theor}
\begin{proof}
The fact that (b) is equivalent to (c) has already been established in Proposition~\ref{ck=nec}. We now check (a) and (b). 

If $\Gamma$ does not admit Necessitation, there is a formula $F$ such that $\Gamma\proves F$, but $\Gamma\not\proves {\bf K}F$\footnote{Here again, we ignore indices $i$ in $\bk_i$ and $R_i$.}. $F$ holds everywhere in ${\it CM}(\Gamma)$. Furthermore, the set $\{\neg {\bf K}F\}$ is consistent over $\Gamma$; consider its maximal consistent extension $u$. Obviously, $\neg {\bf K}F$ holds in $u$ and $F$ holds in $R(u)$ which makes ${\it CM}(\Gamma)$ not fully explanatory and hence not a Kripke model. 

Suppose $\Gamma$ admits Necessitation. To prove that ${\it CM}(\Gamma)$ with $R$ as in (\ref{defR}) is a Kripke model, it suffices to establish the so-called Truth Lemma that membership in $u$ coincides with the truth value at $u$ in the Kripke model $({\it CM}(\Gamma),R)$:  
\[ 		F \in u\ \ \mbox{ iff }\ \ u \forces F\ \mbox{in}\ ({\it CM}(\Gamma),R). \]
Induction on $F$, trivial for atomic $F$'s and standard for Booleans. Consider the case $F={\bf K}X$. 
Let ${\bf K}X \in u$ and $uRv$. Then ${\bf KK}X\in u$, and ${\bf K}X\in v$. 
By {\sf S5}-reflexivity, ${\bf K}X\imp X\in v$, hence $X\in v$. By the IH, $v\forces X$, therefore $u\forces {\bf K}X$.

Put
\[ u^{\bf K}=\{F\mid {\bf K}F\in u\}. 
\]

Let ${\bf K}X\not\in u$. Then $u^{\bf K} \cup \{\neg X\}$ is consistent. Indeed, otherwise $$\Gamma\proves F_1\wedge\ldots\wedge F_n \imp X$$ for some  
$F_i\in u^{\bf K}$. Since $\Gamma$ is closed under Necessitation, $$\Gamma\proves {\bf K}(F_1\wedge\ldots\wedge F_n \imp X).$$ By standard {\sf S5}-reasoning,  $$\Gamma\proves {\bf K}F_1\wedge\ldots\wedge{\bf K} F_n \imp {\bf K}X.$$ 
Since ${\bf K}F_i\in u$, ${\bf K}X$ should be in $u$ as well -- a contradiction. 

Now consider a maximal consistent set $v$ extending  $u^{\bf K}\cup \{\neg X\}$. Obviously, $v\in {\it CM}(\Gamma)$. Since $u^{\bf K}\subseteq v$, $uRv$. Since $\neg X\in v, X\not\in v$, and, by the IH, $v\not\forces X$, which yields $u\not\forces {\bf K}X$. 

\end{proof} 

\begin{example} Consider worlds $A,B,C,D$ from the canonical model for ${\sf S5}(p)$, Section~\ref{single}, Figure 5.
The canonical model of $\Gamma=\{p\}$, ${\it CM}(p)$\footnote{We drop brackets in ${\it CM}(\{p\})$ and similar cases for better readability.}, is the set of worlds at which $p$ holds, i.e., $W=\{A,B\}$, cf. Figure 7. Model ${\it CM}(p)$ is not fully explanatory, not a Kripke model, and is an illustration of Theorem~\ref{main}, since $\{p\}$ is not closed under Necessitation.
\vskip20pt
 
 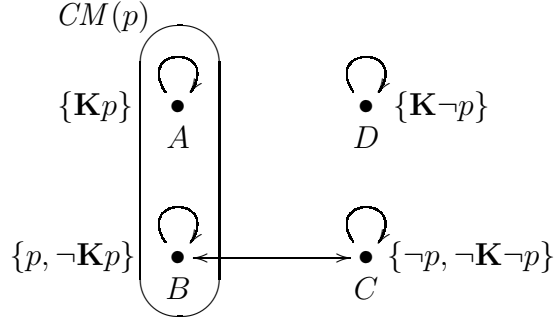
\begin{figure}[!h]
\begin{center}
\mbox{
\begin{xy}
(10,-4)*{C};
(-15,-4)*{B};
(-29,0)*{\mbox{$\{p, \neg\bk p\}$}};
(20,20)*{\mbox{$\{\bk \neg p\}$}};
(10,0)*+{\bullet}="2";
(-15,0)*+{\bullet}="1"; 
{\ar "1";"2"};
{\ar "2";"1"};
{\ar@(ul,ur) "1";"1"};
{\ar@(ul,ur)"2";"2"};
(10,16)*{D};
(-15,16)*{A};
(-26,20)*{\mbox{$\{\bk p\}$}};
(24,0)*{\mbox{$\{\neg p, \neg \bk \neg p\}$}};
(10,20)*+{\bullet}="A";
(-15,20)*+{\bullet}="D"; 
{\ar@(ul,ur) "A";"A"};
{\ar@(ul,ur)"D";"D"};
(-25,32)*{\mbox{${\it CM}(p)$}};
\put(-42,33){\oval(30,110,20)}
\end{xy}
}\end{center}
\caption{Canonical model ${\it CM}(p)$ (in the oval).}
\end{figure}

The canonical model of $\Gamma=\{\bk p\}$ is the set of worlds at which $\bk p$ holds, i.e., $W=\{A\}$. This $\Gamma$ enjoys Necessitation, its canonical epistemic model is fully explanatory, i.e., is a Kripke model. 

The canonical model of $\Gamma=\{\neg \bk p, \neg \bk \neg p\}$. By negative and positive introspection, $\Gamma$ is closed under Necessitation, hence ${\it CM}(\Gamma)$ should also be fully explanatory. Worlds $A$ and $D$ are not compatible with $\Gamma$ and hence are not in ${\it CM}(\Gamma)$. Since none of $B$ and $C$ is such a model (neither is fully explanatory), ${\it CM}(\Gamma)=\{B,C\}$. Indeed, it is an easy exercise to derive $\neg \bk p$ and $\neg \bk \neg p$ in $B$ and $C$. This is a Kripkean situation, i.e., ${\it CM}(\{\neg \bk p, \neg \bk \neg p\})$ is fully explanatory.
\end{example}

\section{Revisiting the Motivating Example: Scaffolding}\label{diss}

We now offer a solution to ``the motivating example," Example~\ref{AB} from Section~\ref{motiv}. 

In Figure 4, we showed a failed attempt to build a Kripke model ${\cal M}_4$ with two worlds $w$ (with {\it heads}, $Q$, and $\neg\bk Q$) and $v$ 
(with {\it tails}, $Q$, and $\neg\bk Q$). An easy argument shows that such a Kripke model does not exist.
Instead of ${\cal M}_4$ we could try to build an epistemic model (Definition~\ref{estate}) ${\cal M}_6$ consisting of two worlds
\begin{itemize}
\item $w$, that extends $a$ from ${\cal M}_2$ and contains $Q$, $\neg\bk Q$; 
\item $v$,  that extends $b$ from ${\cal M}_2$ and contains $Q$, $\neg\bk Q$.
\end{itemize}
\vskip10pt
\begin{figure}[!h]
\begin{center}
\mbox{
\begin{xy}
(10,-3)*{v};
(-15,-3)*{w};
(-35,0)*{Q, \neg\bk Q, \mbox{\it heads},\ldots};
(28,0)*{Q, \neg\bk Q, \mbox{\it tails},\ldots};
(10,0)*+{\bullet}="2";
(-15,0)*+{\bullet}="1"; 
{\ar "1";"2"};
{\ar "2";"1"};
\end{xy}
}\end{center}
\caption{Epistemic model $\mathcal{M}_6$.}
\end{figure}
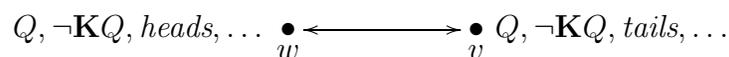
${\cal M}_6$ would be a legitimate epistemic model which is not fully explanatory and hence not a Kripke model. 

There is a subtlety though: the aforementioned specification of ${\cal M}_6$ is not explicit: worlds $w$ and $v$ are not defined, but rather presented by sufficient conditions: for $a$, $b$ from ${\cal M}_2$, 
\begin{equation}\label{conditions}
\mbox{$a\subset w$ and $Q, \neg\bk Q \in w$, $b\subset v$ and $Q, \neg\bk Q \in v$.}
\end{equation}
So, strictly speaking, ${\cal M}_6$ is under-defined. Are there worlds $w$ and $v$ satisfying these conditions? Can we produce explicit examples of such $w$ and $v$? One observation is immediate, neither $w$ nor $v$ is uniquely defined, e.g., though conditions (\ref{conditions}) yield $\neg\bk(\mbox{\it heads})$ and $\neg\bk Q$, they yield neither 
$\bk (\mbox{\it heads}\vee Q)$ nor its negation\footnote{Easy counter-models for both.}. 
We now answer these legitimate questions and provide a final version of an epistemic model for the story from Example~\ref{AB}. 
First, we consider a Kripke model ${\cal M}_7$ (identical to ${\cal M}_3$):

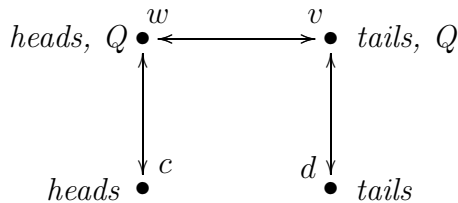
\begin{figure}[!h]
\begin{center}
\mbox{
\begin{xy}
(8,3)*{v};
(-13,3)*{w};
(7,-17)*{d};
(-12,-17)*{c};
(-25,0)*{\mbox{\it heads, Q}};
(20,0)*{\mbox{\it tails, Q}};
(-23,-20)*{\mbox{\it heads}};
(17,-20)*{\mbox{\it tails}};
(10,0)*+{\bullet}="2";
(-15,0)*+{\bullet}="1"; 
(10,-20)*+{\bullet}="4";
(-15,-20)*+{\bullet}="3";
{\ar "1";"2"};
{\ar "2";"1"};
{\ar "1";"2"};
{\ar "2";"1"};
{\ar "1";"3"};
{\ar "3";"1"};
{\ar "2";"4"};
{\ar "4";"2"};
\end{xy}
}\end{center}
\caption{Model $\mathcal{M}_7$.}
\end{figure}
\noindent
It is easy to check that worlds $a$ and $b$ from ${\cal M}_2$ hold at $w$ and $v$ respectively. Furthermore, 
\[ w,v\forces Q, \neg\bk Q, \]
so $w$ and $v$ from $\mathcal{M}_7$ satisfy conditions~(\ref{conditions}) and they are explicitly defined in $\mathcal{M}_7$.

It now remains to carve the desired epistemic model $\mathcal{M}_8$ out of $\mathcal{M}_7$ with its fictional states $c$ and $d$: 
\vskip10pt
\begin{figure}[!h]
\begin{center}
\mbox{
\begin{xy}
(-47,0)*{\mbox{\it $\mathcal{M}_8:$}};
(8,3)*{v};
(-13,3)*{w};
(7,-17)*{d};
(-12,-17)*{c};
(-25,0)*{\mbox{\it heads, Q}};
(20,0)*{\mbox{\it tails, Q}};
(-23,-20)*{\mbox{\it heads}};
(17,-20)*{\mbox{\it tails}};
(10,0)*+{\bullet}="2";
(-15,0)*+{\bullet}="1"; 
(10,-20)*+{\bullet}="4";
(-15,-20)*+{\bullet}="3";
{\ar "1";"2"};
{\ar "2";"1"};
{\ar "1";"2"};
{\ar "2";"1"};
{\ar "1";"3"};
{\ar "3";"1"};
{\ar "2";"4"};
{\ar "4";"2"};
\put(-10,1){\oval(200,30,20)}
\end{xy}
}\end{center}
\caption{Epistemic model $\mathcal{M}_8$, in the oval.}
\end{figure}
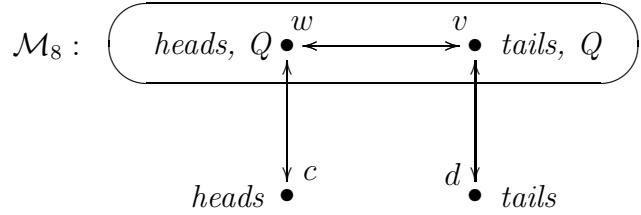

We offer ${\cal M}_8$ as a superior description of Ann's states from Example~\ref{AB} compared to Kripke-style models ${\cal M}_2$ and ${\cal M}_3$.  Ann considers possible only two states $w$ and $v$, in which both $Q$ and  $\neg\bk Q$ hold. Since Ann in unaware of $Q$, she does not know $\mathcal{M}_8$ in full. Model $\mathcal{M}_8$ is not fully explanatory: $Q$ holds everywhere in  $\mathcal{M}_8$ but is not known.

What is the role of technical states $c$ and $d$? We use them in model $\mathcal{M}_7$ as ``scaffolding," to properly define possible states $w$ and $v$, and then remove $c$ and $d$ as auxiliary material to obtain the true epistemic model $\mathcal{M}_8$. 

In view of embedding Theorem~\ref{two}, such a {\it scaffolding method} can be offered as universal: in order to define an epistemic model with possible worlds $W$, we consider a convenient Kripke model with states $\widetilde{W}\supset W$ thus using the power of Kripke models to define states, and then cut off auxiliary states $\widetilde{W}\setminus W$ to get the desired epistemic model.

We now introduce Bob's knowledge to obtain the final version of the epistemic model. We assume that Bob knows $Q$, but not the results of a coin toss. 
Again, we ``scaffold."
\begin{figure}[!h]
\begin{center}
\mbox{
\begin{xy}
(10,3)*{v};
(-15,3)*{w};
(-3,3)*{R_1,R_2};
(7,-17)*{d};
(-12,-17)*{c};
(-25,0)*{\mbox{\it heads, Q}};
(-18,-10)*{R_1};
(20,0)*{\mbox{\it tails, Q}};
(13,-10)*{R_1};
(-23,-20)*{\mbox{\it heads}};
(17,-20)*{\mbox{\it tails}};
(10,0)*+{\bullet}="2";
(-15,0)*+{\bullet}="1"; 
(10,-20)*+{\bullet}="4";
(-15,-20)*+{\bullet}="3";
{\ar "1";"2"};
{\ar "2";"1"};
{\ar "1";"2"};
{\ar "2";"1"};
{\ar "1";"3"};
{\ar "3";"1"};
{\ar "2";"4"};
{\ar "4";"2"};
\end{xy}
}\end{center}
\caption{Kripke model $\mathcal{M}_9$.}
\end{figure}
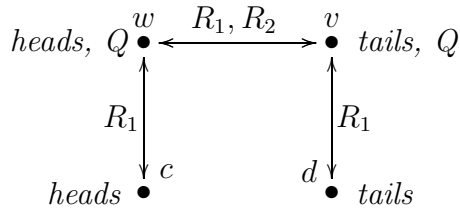
We take a Kripke model $\mathcal{M}_9$ as in Figure 11, with $\widetilde{W}=\{w,v,c,d\}$, $R_1(w)=\widetilde{W}$, $R_2(w)=\{w,v\}$, and $\Vdash$ as shown. 
At $w$, Ann and Bob do not know {\it heads} and {\it tails}, Bob knows $Q$ since $Q$ holds in $R_2(w)$, Ann does not know $Q$ since $Q$ fails in $c$ accessible from $w$, etc. 

The very last step of building an adequate model for Example~\ref{AB}, $\mathcal{M}_{10}$, is removing auxiliary states $c$ and $d$, which are considered possible by neither Ann nor Bob, and maintaining states $w$ and $v$ as they were in $\mathcal{M}_9$.
\vskip7pt
\begin{figure}[!h]
\begin{center}
\mbox{
\begin{xy}
(-47,0)*{\mbox{\it $\mathcal{M}_{10}$:}};
(10,3)*{v};
(-15,3)*{w};
(-2,3)*{R_1,R_2};
(7,-17)*{d};
(-12,-17)*{c};
(-25,0)*{\mbox{\it heads, Q}};
(-18,-10)*{R_1};
(20,0)*{\mbox{\it tails, Q}};
(13,-10)*{R_1};
(-23,-20)*{\mbox{\it heads}};
(17,-20)*{\mbox{\it tails}};
(10,0)*+{\bullet}="2";
(-15,0)*+{\bullet}="1"; 
(10,-20)*+{\bullet}="4";
(-15,-20)*+{\bullet}="3";
{\ar "1";"2"};
{\ar "2";"1"};
{\ar "1";"2"};
{\ar "2";"1"};
{\ar "1";"3"};
{\ar "3";"1"};
{\ar "2";"4"};
{\ar "4";"2"};
\put(-10,2){\oval(200,30,20)}
\end{xy}
}\end{center}
\caption{Epistemic model $\mathcal{M}_{10}$, in the oval.}
\end{figure}
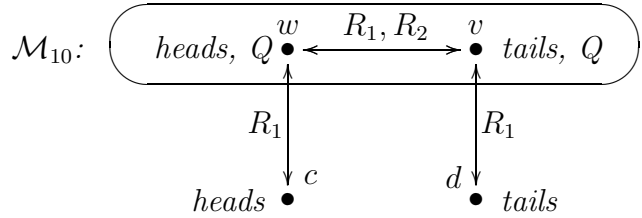
Let us read model $\mathcal{M}_{10}$. Both Ann and Bob deem states $w$ and $v$ possible, but differ in knowledge of the truth assignments. Bob knows the whole model $\mathcal{M}_{10}$:  $Q$ holds at both $w$ and $v$, {\it heads} holds at $w$, and {\it tails} holds at $v$. Ann knows that {\it heads} holds at $w$ and {\it tails} holds at $v$, but is unaware of $Q$ and has no idea that $Q$ holds in $\mathcal{M}_{10}$. 

\section{Findings and Suggestions}
 We suggest streamlining the foundations of epistemic modeling. If we start from a given set of possible worlds and define an accessibility/indistinguishability relation between them in the standard way, the result might not be a Kripke model: only those structures that are fully explanatory correspond to Kripke models. These assumptions should be made explicit.
 
 The fully explanatory property is a propositional formalization of common knowledge of the model, which has been tacitly assumed in epistemic modal logic. Loosely speaking, 
 \[ \mbox{\it Kripke Models are Epistemic Models commonly known to the agents.} \]

We have sketched a basic theory of epistemic models constituting a broader and more expressive class than  traditional Kripke models. We see perspectives of using epistemic models in situations with partial and asymmetric knowledge of the model, and in other epistemic scenarios in which a model itself in not necessarily common knowledge. 

Our approach does not deny Kripke models. Moreover, within the framework of Theorem~\ref{two},  the latter can be used as a convenient technical tool for working with epistemic models, cf. ``scaffolding method" from Section~\ref{diss}.

\section{Acknowledgements}
The author is grateful to Vladimir Krupski, Elena Nogina, and Tudor Protopopescu for helpful suggestions. 
Special thanks to Karen Kletter for editing and proofreading this text.

\end{document}